\documentclass[11pt,a4paper]{article}

\usepackage[width=420pt]{geometry}
\usepackage{amsmath}
\usepackage{amssymb}
\usepackage{amsthm}
\usepackage{mathtools}
\usepackage{color}
\usepackage{authblk}
\usepackage{hyperref}
\usepackage{enumerate}
\usepackage{tikz}
\usepackage[basic]{mnotes}
\usepackage[capitalise]{cleveref}
\usepackage{cite}

\makeatletter
\newcommand{\refcheckize}[1]{%
  \expandafter\let\csname @@\string#1\endcsname#1%
  \expandafter\DeclareRobustCommand\csname relax\string#1\endcsname[1]{%
    \csname @@\string#1\endcsname{##1}\wrtusdrf{##1}}%
  \expandafter\let\expandafter#1\csname relax\string#1\endcsname
}
\makeatother

\allowdisplaybreaks

\renewcommand{\mid}{\;:\;}

\DeclareMathOperator{\R}{\mathbb{R}}
\DeclareMathOperator{\bbS}{\mathbb{S}}
\DeclareMathOperator{\N}{\mathbb{N}}
\DeclareMathOperator{\eps}{\varepsilon}
\DeclareMathOperator{\ubd}{\overline{\dim}_{\mathrm{B}}}
\DeclareMathOperator{\umbd}{\overline{\dim}_{\mathrm{MB}}}
\DeclareMathOperator{\lbd}{\underline{\dim}_{\mathrm{B}}}
\DeclareMathOperator{\bdd}{\dim_{\mathrm{B}}}

\DeclareMathOperator{\argmin}{argmin}
\DeclareMathOperator{\cT}{\mathcal{T}}
\DeclareMathOperator{\cL}{\mathcal{L}}
\DeclareMathOperator{\Homeo}{Homeo}

\DeclareMathOperator{\diam}{diam}
\DeclareMathOperator{\id}{id}
\DeclareMathOperator{\cl}{cl}

\newtheorem{theorem}{Theorem}[section]
\newtheorem{lemma}[theorem]{Lemma}
\newtheorem{proposition}[theorem]{Proposition}
\newtheorem{corollary}[theorem]{Corollary}

\theoremstyle{definition}
\newtheorem{definition}[theorem]{Definition}
\newtheorem{example}[theorem]{Example}

\numberwithin{equation}{section}

\title{On continuum real trees of circle maps and their graphs}
\author[$\star$]{Maik Gr\"oger}
\author[$\dagger$]{Sascha Troscheit}
\affil[$\star$]{\small Faculty of Mathematics and Computer Science,
  Jagiellonian University in Krak\'ow, Poland.

Email: maik.groeger@im.uj.edu.pl\vspace{1em}}
\affil[$\dagger$]{\small Mathematical Sciences Research Unit, University of Oulu, PO Box 8000, 90014 Oulu,
  Finland.

Email: arxiv@troscheit.eu}

\begin{document}

\renewcommand*{\thefootnote}{\fnsymbol{footnote}}

\maketitle
\begin{abstract}
  The Brownian continuum tree was extensively studied in the 90s as a universal random metric space.
  One construction obtains the continuum tree by a change of metric from an excursion function (or continuous
  circle mapping) on $[0,1]$.
  This change of metric can be applied to all excursion functions, and generally to continuous circle mappings.
  In 2008, Picard proved that the dimension theory of the tree is connected to its associated
  contour function:
  the upper box dimension of the continuum tree coincides with the variation index of the contour function.

  In this article we give a short and direct proof of Picard's theorem through the study of packings.
  We develop related and equivalent notions of variations and variation indices and study their basic properties.
  Finally, we link the dimension theory of the tree with the dimension theory of the graph of its
  contour function.
  \footnotetext[1]{The research leading to these results has received funding from the Norwegian
Financial Mechanism 2014-2021 via the POLS grant no.~2020/37/K/ST1/02770.}
  \footnotetext[2]{Research initially supported by
    FWF (Austrian Science Fund) Meitner grant M-2813 and later by the European Research Council
  Marie Sk\l{}odowska-Curie Personal Fellowship \#101064701.}
\end{abstract}

\renewcommand*{\thefootnote}{\arabic{footnote}}


\section{Introduction}
Let $f:[0,1]\to [0,\infty)$ be a continuous map satisfying
$f(0)=f(1)=0$ and $f(x)\geq0$ for all $x\in [0,1]$.
We say that $f$ is an \emph{excursion function} and we can associate with it a real tree through a
change of metric.
In the context of a real tree, it is also known as a \emph{contour function} of the tree.
Define the pseudometric $d_f$ by
\[
  d_f(x,y)  = f(x)+f(y) - 2\min_{z\in[x,y]}f(z)
\]
and $d_f(y,x)=d_f(x,y)$ on all $x\leq y$ in $[0,1]$. 
This gives rise to a bona fide metric space $([0,1]_\sim, d_f)$ by identifying points in $[0,1]$ respective to $d_f$.
As a convention, we will denote the representative of the equivalency class $[x]$ for $x\in [0,1]$
by $\max\left\{ y\in[0,1] \mid y\in [x] \right\}$.
It is not too hard to see that the resulting metric space is a real tree by noticing that for all
$x\in [0,1]$ the path $\gamma_x(t)
= \max \left\{ s \in[0, x] \mid f(s) = t   \right\}$
is a geodesic from the root $0\sim 1$ to $x$ in $[0,1]_\sim$ and that for every $x,y\in [0,1]$ the
greatest common ancestor is (in the equivalency class of) $z=\argmin_{s\in[x,y]}f(s)$ and both geodesics $\gamma_x,\gamma_y$
traverse through (the equivalency class of) $z$, i.e.\ $\gamma_x(f(z)) \sim \gamma_y(f(z))$.
We shall refer to these tree spaces by $\cT_f$, see \cref{fig:smoothtree} for an illustrative example.
We note here that it is not necessary to stipulate $f(x)>0$ for all $x\in(0,1)$. In fact, we may
extend the family of functions under consideration to any continuous (real) functions on the circle
$[0,1]_{0\sim 1} = \bbS^1$, as a suitable translation makes it an excursion with potentially
multiple zeros. We will assume throughout that our circle mappings are non-constant to
avoid trivialities. We also remark that in many cases excursion functions are defined to be positive
in the interior of $(0,1)$. This avoids the tree splitting at the root node but we will not make
this assumption to work in slightly higher generality.
\begin{figure}[t]
  \begin{center}
  \begin{tikzpicture}
    \node at (0,0) {\includegraphics[width=.5\textwidth]{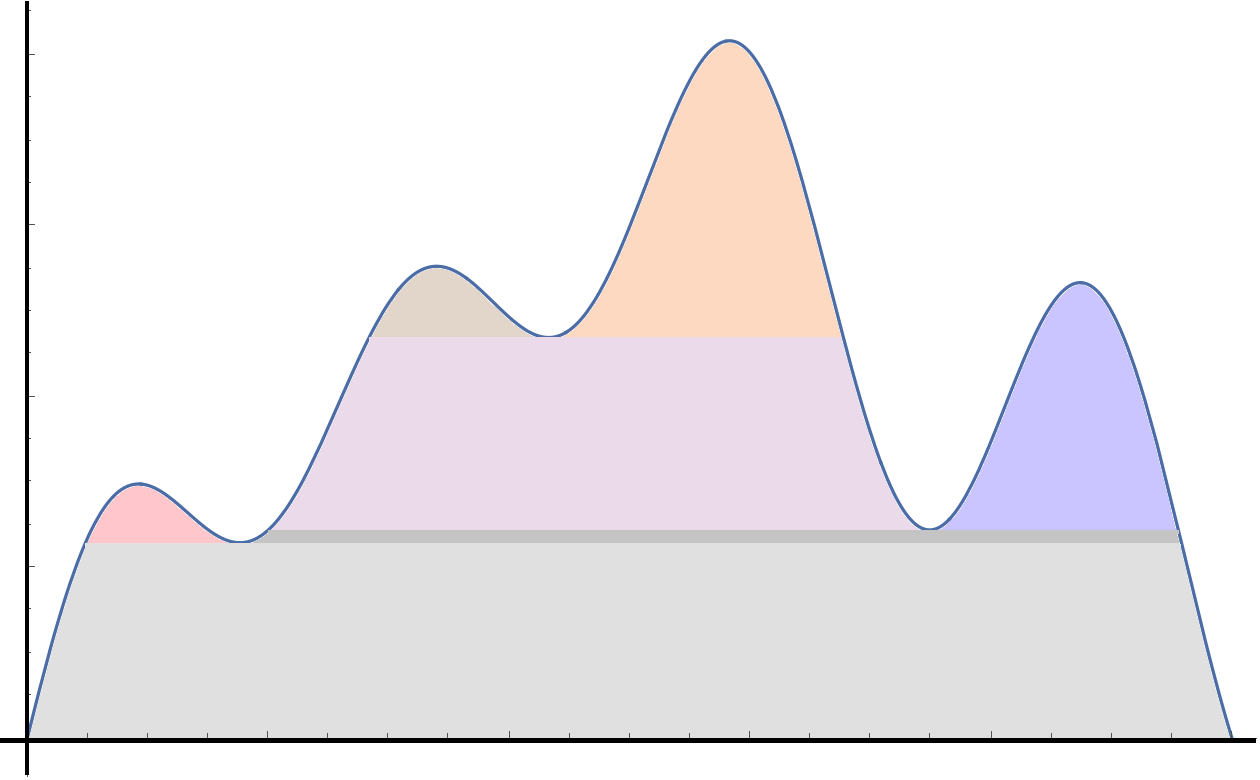}};
    \node at (1.7,1) {$f(x)$};
    \node at (5,0) {$\longmapsto$};
    \node at (5,-0.3) {$\pi$};
    \node at (-3.7,-2.3) {$0$};
    \node at (3.6,-2.3) {$1$};
    \node at (8,0) {\includegraphics[width=0.2\textwidth]{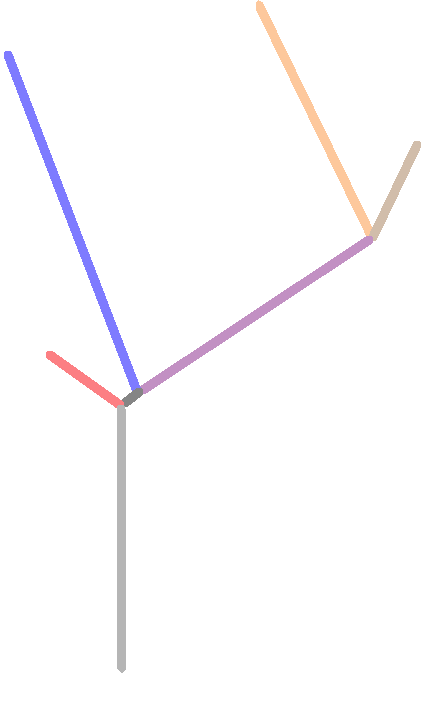}};
    \node at (8,-1.5) {$\cT(f)$};
  \end{tikzpicture}
\end{center}
  \caption{The tree space associated with a very smooth function. The identified regions in the
  tree are captured by colour coding.}
  \label{fig:smoothtree}
\end{figure}
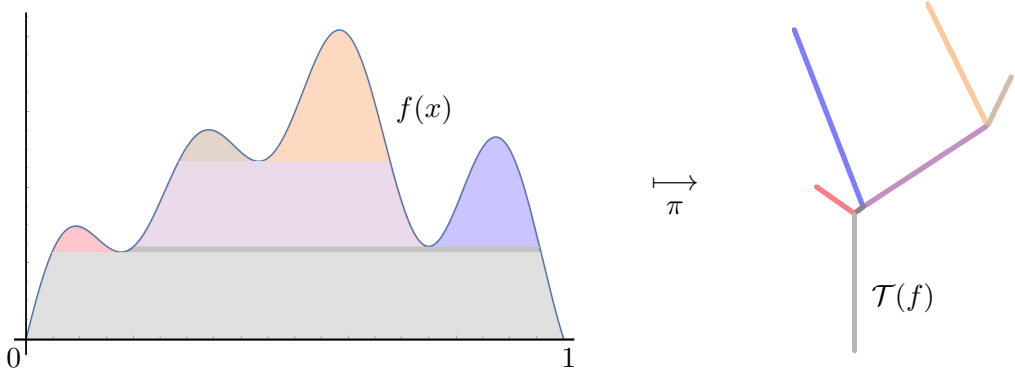

Real trees appear in a variety of contexts such as probability theory and geometric group theory.
Consider a Brownian excursion, that is a Gaussian processes on $[0,1]$ conditioned
on being an excursion function.
The resulting random tree is known as the Brownian continuum random tree (CRT) and 
is an important universal random metric space, see \cref{fig:bcrt}.
The CRT appears in a multitude of settings and has several natural and equivalent definitions that were
first explored by Aldous in the 1990s \cite{ Aldous91a, Aldous91,Aldous93}.
Picard \cite{Picard08} noticed that the $p$-variation of the contour function $f$ is
linked to the dimension theory of the tree: The variation index of $f$ coincides with the upper
box dimension of the tree. 
Since the variation index is indicative of the H\"older exponent of $f$ (up to time change), this
links analytical properties of the function with the dimension theory of the tree.

As such the tree representation of continuous functions is an important tool in the study of their 
analytical properties, see \cite{Broutin21} and references therein. 
Interpreting this correspondence slightly liberally, we may think of the tree space as a 
``dual space''.
However, strictly speaking this is not true.
Notice that for any homeomorphism 
$\tau :\bbS^1 \to\bbS^1$ and continuous circle mapping
$f:\bbS^1\to \R$ we have $\cT_f = \cT_{f\circ\tau}$. We shall refer to such homeomorphisms as
\emph{time changes}. This is not the only way that two distinct functions can have the same
associated tree space as particular permutations can have the same effect. 

\paragraph{Organisation.}
The goal in this article is to explore the dimension theory of tree spaces more thoroughly.
We will start by giving definitions in \cref{sect:definitions}
and provide a new short proof of Picard's theorem using basic
notions in dimension theory in \cref{sect:picard}.
Crucially, our proof avoids the use of setting up integration along the tree and avoids averaging,
or truncation, of the fine detail of the tree directly.

\smallskip
Our new proof of Picard's theorem  motivates us to define new concepts of variations that correspond to 
other dimension-theoretic notions.
This is done in \cref{sect:morevariations}, where we 
define a discretised version of the $p$-variation $V_r^p(f)$. Taking the upper and lower limit gives
two new notions of variation that we coin the upper and lower variation content of $f$, denoted by
$\overline{V}^p(f)$ and $\underline{V}^p(f)$, respectively.
Each of these notions comes with its own notion of variation index that more closely resembles
dimensions defined through measures or contents. 

We show that the upper variation content is bounded above by the variation, $\overline{V}^p(f) <
V^p(f)$, but that its indices coincide. Thus, the upper variation content is a finer measure of
smoothness as there exist functions for which $0<\overline{V}^p(f)< V^p(f) = \infty$.
The notions of upper and lower variation index give a direct correspondence to the upper and lower
box dimension of the associate real tree as we will show in \cref{thm:comparable}. In fact,
\cref{thm:comparable} establishes a direct link to packings in the associated real tree, with
\cref{thm:doublingComparable} showing that for trees that are doubling metric spaces, the
discretised variation $V_r^p(f)$ and $N_r(\cT(f))$ are comparable.


\smallskip
In \cref{sect:shmerkin}, we will explore the link between the upper box dimension of
the graph $\Gamma_f = \left\{ (x,f(x)) : x\in \bbS^1 \right\}$
of $f:\bbS^1 \to \R$ and the upper box dimension of $\cT_f$, which is equal to
the variation index of $f$.
The following inequality holds for all time changes $\kappa$:
\[
  \ubd \Gamma_{f\circ \kappa}\leq 2-\frac{1}{\ubd\cT_f}
\]
and follows from Picard's theorem and
the well-known fact that any $\alpha$-H\"older function $f$ satisfies $\ubd\Gamma(f) \leq 2-\alpha$,
see \cite[Corollary 11.2]{FalconerBible}
Inspired by a question of P.\ Shmerkin, we will prove a ``variational principle''
and show that the upper bound is achieved by some time changes $\tau$.
That is, for all $f:\bbS^1\to\R$ there exists time change $\tau$ such that 
\[
  \ubd \Gamma_{f\circ \tau}= 2-\frac{1}{\ubd\cT_f}.
\]
We will further show that there exists a time change $\underline\tau$ for which the
modified upper box dimension provides a lower bound:
\[
  \ubd \Gamma_{f\circ\underline\tau}\geq 2-\frac{1}{\umbd\cT_f}.
\]
However, it is not true that this holds for any time change, and hence does not give a variational
principle as for the upper bound. 

\smallskip
Finally, in \cref{sect:examples}, we will construct several examples of circle mappings that illustrate results throughout
the article.

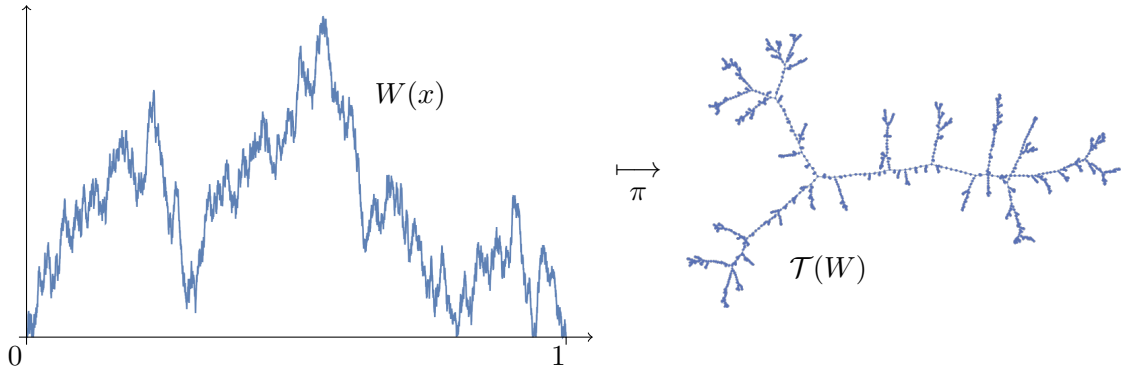
\begin{figure}[t]
  \begin{center}
    \begin{tikzpicture}
    \node at (0,0) {\includegraphics[width=.5\textwidth]{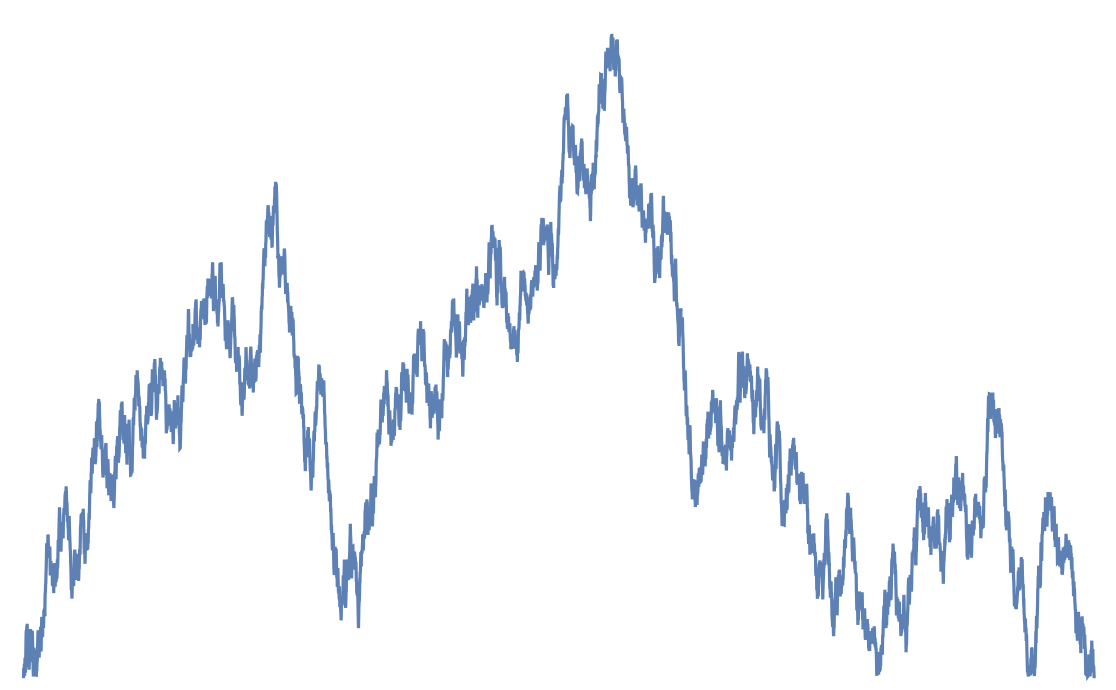}};
      \node at (1.5,1) {$W(x)$};
      \node at (4.5,0) {$\longmapsto$};
      \node at (4.5,-0.3) {$\pi$};
      \draw[->]  (-3.55,-2.3) -- (-3.55,2.2);
      \draw[->] (-3.65,-2.2) -- (3.9,-2.2);
      \draw (3.55,-2.2) -- (3.55,-2.3);
      \node at (-3.7,-2.45) {$0$};
      \node at (3.45,-2.45) {$1$};
      \node at (8,0) {\includegraphics[width=0.4\textwidth]{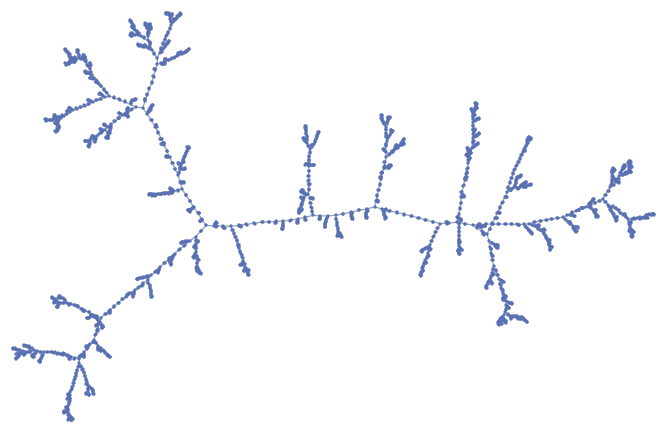}};
      \node at (7,-1.3) {$\cT(W)$};
    \end{tikzpicture}
  \end{center}
  \caption{A Brownian excursion realizing the Brownian continuum random tree.}
  \label{fig:bcrt}
\end{figure}

\paragraph{Acknowledgements.}
The authors wish to express their gratitude to Daniel Meyer, Tony Samuel, and Pablo Shmerkin with whom some
of the questions in this article were extensively discussed. 
\cref{sect:shmerkin} in particular was inspired by a question of Pablo Shmerkin.
The authors further wish to thank Steffen Rohde, Mariusz Urba\'nski, and Meng Wu for interesting
conversations and talks related to this topic.


\section{Definitions: Real trees of circle mappings}
\label{sect:definitions}
The real tree of a function is defined through excursion functions.
\begin{definition}
  An \textbf{excursion function} is any real-valued, continuous function $f$ on $[0,1]$, such that
  $f(0)=0=f(1)$ with $f(x) \geq 0$ for all $x\in[0,1]$. 
\end{definition}
By changing the metric on $[0,1]$ we 
obtain a real rooted tree from every excursion function.
\begin{definition}
  Let $f:[0,1]\to[0,\infty)$ be an excursion function. Then, for all $0\leq x \leq y \leq 1$, we set
  \[
    d_f(x,y) = f(x)+f(y) - 2\min_{z\in[x,y]}f(z)
    \quad\text{and}\quad
    d_f(y,x) = d_f (x,y).
  \]
  The function $d_f:[0,1]^2\to [0,\infty)$ is a pseudometric and, identifying equivalence
  classes, $\cT(f) = ([0,1]/_\sim, d_f)$ is a metric space called the \textbf{real (rooted) tree} of
  $f$.
\end{definition}
We will write $\pi:[0,1]\to[0,1]/_\sim$ for the map that projects the Euclidean space $([0,1],|.|)$
into the real tree space $\cT(f)$.
Note that restricting to excursion functions is somewhat arbitrary. 
We can associate a real
tree to any continuous function $f\in C(\bbS^1)$ on the circle $\bbS^1$ by translating the function to
$\widetilde{f}(x)=f(x-a)+b$ with $a=\argmin_x f(x)$ and $b=f(a)$ such that 
$0=\widetilde{f}(0)=\widetilde{f}(1)\leq \widetilde{f}(x)$ for all
$x\in\bbS^1$.
This extension still captures all relevant analytical properties that we are interested in.
\begin{definition}
  The real tree $\cT(f)$ of a continuous circle mapping $f:\bbS^1\to\bbS^1$ is the real
  tree of the excursion function $\widetilde{f}:\bbS^1\to [0,\infty)$ given by $x\mapsto f(x-a)+f(a)$,
  where $a=\argmin_x f(x)$.
\end{definition}
In the remainder of this article we may freely switch between the notions of continuous circle
mappings and excursions through this identification.

\paragraph{Remark.}
The projection $\pi$ resulting from the change of metric above is a mapping from the continuous circle mappings
$f\in C(\bbS^1)$ onto the space of bounded real trees.
Here we understand $\bbS^1 = [0,1]/_\sim$ where $0\sim 1$.
This projection fails to be injective. Indeed, it is not difficult to see that 
for any homeomorphism $\tau\in\Homeo(\bbS^1)$ we have
$\pi(f) = \pi(f\circ\tau)$.
However, we may quotient out functions giving the same real tree and consider equivalency
classes.


\subsection{Bounded variation and \texorpdfstring{$p$}{p}-variation}
The $p$-variations, and the variation in particular, are an important collection of seminorms
of real-valued functions, see for example the text book \cite{LeoniBook} on its use in functional
analysis as well as \cite{Chistyakov04, Chistyakov98, Porter05}.
This family of variations can be considered as a measure of regularity and are linked to the H\"older continuity
of functions. 
\begin{definition}
  Let $p\geq 1$. The \textbf{$p$-variation} of a map $f:[0,1]\to \R$ is given by 
  \[
    \|f\|_p = \sup\left\{ \sum_{i=1}^{n-1}|f(x_{i}) - f(x_{i+1})|^p\mid \{x_{i}\}_{i=1}^n\text{ is a
    finite partition of }[0,1]  \right\}^{1/p}
  \]
  where the supremum is over all 
  finite partitions $\left\{ x_i \right\}_{i=1}^n$ with $x_i\in [0,1]$ and $x_i < x_{i+1}$.
\end{definition}
However, in the following we will not be interested in the variations as norms and we are mostly
concerned whether the variation is positive or finite. Thus, we may ignore the last power and write
$V^p(f) = \|f\|_{p}^{p}$. We will also refer to the quantity $V^p$ as the \textbf{$p$-variation}.

The $p$-variation itself may have a critical exponent known as the variation index. 
\begin{definition}
  Let $f:[0,1]\to\R$. The \textbf{variation index} of $f$ is given by
  \[
    I(f) = \sup\left\{ p\geq 1 \mid V^p(f) =\infty \right\}.
  \]
\end{definition}
The definition of the variation index seems natural from the viewpoint of dimension theory, cf.\ the
definition of the Hausdorff dimension via the value of the $s$-Hausdorff measure.
One would perhaps also expect that 
\[
  \sup\left\{ p\geq 1 \mid V^p(f) =\infty \right\} = \inf\left\{ p\geq 1 \mid V^p(f)
  =0 \right\}
\]
as is often the case for dimensions. However, this is not true, as for any non-constant function
$f$, every $p$-variation is bounded below by $\max_{x,y\in[0,1]}|f(x)-f(y)|^p> 0$.

Recall that $f$ is said to be \textbf{$\alpha$-H\"older} if there exists $C>0$ such that
\[
  |f(x)-f(y)| \leq C |x-y|^\alpha
\]
for all $x,y\in[0,1]$.
The variation index of a function is further related to its H\"older exponent
as the following well-known result shows.
\begin{proposition}\label{thm:holdervariation}
  Let $f\in C(\bbS^1)$. The following two statements hold:
  \begin{enumerate}[(1)]
    \item If $f$ is $\alpha$-H\"older, then $V^{1/\alpha}(f\circ\tau)<\infty$ for any homeomorphism $\tau:\bbS^1\to\bbS^1$.
    \item If $V^p(f) < \infty$, then there exists a homeomorphism $\tau:\bbS^1 \to \bbS^1$ such that $f\circ \tau$ is $1/p$-H\"older.
  \end{enumerate}
\end{proposition}
While the proof is standard, we provide it for completeness and as the basis for proofs that occur later in this paper.
\begin{proof}
  We first proof statement (1).
  Let $f$ be an $\alpha$-H\"older excursion function\footnote{Here and elsewhere, we will usually assume without loss
    of generality that $f$ is an excursion function on $[0,1]$. This often simplifies arguments by
  giving us an unambiguous ordering of the set, while not affecting the results.}, i.e.\ 
  \[
    |f(x)-f(y)| < C |x-y|^\alpha.
  \]
  Let $\{x_1,\dots,x_n\}\subset [0,1]$ be such that $x_i < x_{i+1}$. Then, for $p=1/\alpha$,
  \begin{align*}
    \sum_{i=1}^{n-1} |{f}(x_i) - {f}(x_{i+1})|^p 
    &\leq
    C^p \sum_{i=1}^{n-1} |x_{i} - x_{i+1} | \leq C^{1/\alpha}.
  \end{align*}
  Since the partition $\{x_i\}_{i=1}^n$ was arbitrary, we have $V^p(f) = V^{1/\alpha}(f) \leq C^{1/\alpha}<\infty$ as required.

  \smallskip
  For statement (2) we define $h(x)=V^p(f|_{[0,x]})$ for all $x\in[0,1)$ and 
  \[
    f|_{[x,y]}(z) = 
    \begin{cases}
      f(x) & \text{ for all }z\in[0,x),\\
      f(z) & \text{ for all }z\in[x,y],\\
      f(y) & \text{ for all }z\in(y,1).
    \end{cases}
  \]
  Notice that $h$ is continuous and non-decreasing and that $h(x)\to V^p(f)$ as $x\to 1$.
  Write $v = V^p(f)$. Then, $h_1(x) = \tfrac{1}{2v}h(x) + \tfrac{x}{2}$ is a strictly increasing
  continuous function such that $h_1(0)=0$ and $h_1(x)\to 1$ as $x\to 1$.
  Hence, $h_1$ is a homeomorphism on $\bbS^1$ and so is its inverse $\tau = h_1^{-1}$.

  It remains to show that $f\circ \tau$ is $\tfrac1p$-H\"older.
  Let $x,y\in[0,1)$ be distinct and assume without loss of generality that $x<y$.
  Then, 
  \begin{align*}
    |f(x) - f(y)|^p 
    &\leq |V^p(f|_{[x,y]})|
    \\
    &\leq |h(y)-h(x)|
    \\
    &=|2v(h_1(y)-y/2) - 2v(h_1(x)-x/2)|
    \\
    &=|2v(h_1(y)-h_1(x))-v(y-x)|
    \\
    &\leq 2v|h_1(y)-h_1(x)|.
  \end{align*}
  Writing $a=h_1(x)$ and $b=h_1(y)$, we obtain
  \[
    |f\circ \tau (a) - f\circ \tau(b)| \leq (2v)^{1/p}\cdot |a - b|^{1/p}.\qedhere
  \]
\end{proof}

\subsection{Dimension theory}
In the final part of this section we give a brief introduction of dimension theory, as well as
define some common dimensions used in metric geometry.
More detail on this subject can be found in the introductory books
\cite{FalconerBible,FalconerTechniques,MattilaGMT}

\subsubsection{The box dimension}
The box dimension (also known as box-counting or Minkowski dimension) of a metric space 
arises by considering coverings by sets of
comparable size. 
Let $N_r(X)$ be the minimal number of sets of diameter at most $r$ that are
required to cover a metric space $X$.
Recall that $N_r(X)$ is always finite if $X$ is totally bounded.
The box dimension establishes the relationship between number of covers and their size by $N_r(X)
\approx r^{-\bdd X}$.
\begin{definition}
  Let $X$ be a totally bounded metric space. The \textbf{upper} and \textbf{lower box dimension} are defined by
  \[
    \ubd X = \limsup_{r\to 0} \frac{\log N_r(X)}{-\log r}
    \quad
    \text{and}
    \quad
    \lbd X = \liminf_{r\to 0} \frac{\log N_r(X)}{-\log r}.
  \]
  We say that $X$ has \textbf{box dimension} $\bdd X = s$ if $\ubd X = \lbd X=s$.
\end{definition}

There are a number of equivalent notions of box dimension and since we will require these
equivalent notions for arbitrary metric spaces and most text books only consider Euclidean space,
we will spend some time showing they are equivalent.

Let $(X,d)$ be a metric space. We denote the open $r$-ball centred at $x\in X$ 
by $B(x,r) = \left\{ y\in X \mid d(x,y)<r \right\}$, the closed
$r$-ball by $\overline{B}(x,r) = \cl(B(x,r))$ and say that a collection of open $r$-balls $(B_i)_i$
is an \textbf{$r$-cover} of $X$ if $X=\bigcup_i B_i$. Similarly, we say that a collection of closed $r$-balls
$(\overline{B}_i)_i$ is an \textbf{$r$-packing} of $X$ if the collection is pairwise disjoint.
A subset $A\subseteq X$ is said to be \textbf{$r$-separated} if $d(x,y)>r$ for all distinct $x,y\in A$.
\begin{proposition}
  Let $X$ be a totally bounded metric space. 
  The upper and lower box dimension may equivalently be defined by replacing $N_r(X)$ with
  \begin{enumerate}[(1)]
    \item $N_r^{(1)}(X)$, the minimal cardinality of an $r$-cover of $X$.
    \item $N_{r}^{(2)}(X)$, the maximal cardinality of an $r$-separated subset of $X$.
    \item $N_r^{(3)}(X)$, the maximal cardinality of an $r$-packing of $X$.
  \end{enumerate}
\end{proposition}
\begin{proof}
  Consider all points in the complement of a maximal $r$-packing of $X$. For each such point
  $x$ the distance to the nearest centre of a packing ball must be less than $2r$, since otherwise $B(x,r)$
  does not intersect any other ball, contradicting maximality of the packing. Hence $x$ is covered
  by enlarging that closest ball by a factor of $3$ (for open balls). But since $x$ was arbitrary,
  enlarging the $r$-packing of $X$ to open $3r$-balls gives a cover of $X$ by open $3r$-balls,
  hence $N_{3r}^{(1)}(X) \leq N_r^{(3)}(X)$.
  By taking the centres of the balls in a maximal $r$-packing of $X$, we get an $r$-separated
  set and  $N_{r}^{(2)}(X)\geq N_{r}^{3}(X)$.
  Any cover of sets with diameter at most $r/2$ can contain at most one element of an $r$-separated
  set. Hence $N_{r/2}(X)\geq N_r^{(2)}(X)$.
  Any cover by open balls of radius $r/2$ is also a cover by sets of diameter at most $r$. 
  Hence $N_r(X) \leq N_{r/2}^{1}(X)$.

  We have
  \[
    N_{3r}^{(1)}(X) \leq N_{r}^{(3)}(X) \leq N^{(2)}_{r}(X) \leq N_{r/2}(X) \leq
    N_{r/4}^{(1)}(X).
  \]
  Since for all $c_1,c_2 >0$,
  \(
    \frac{-\log c_1r}{-\log c_2r} \to 1
  \)
  as $r\to 0$ we find that 
  \[
    \frac{N_{3r}^{(1)}(X)}{-\log 3r}
    \leq C_1(r) \frac{ N_{r}^{(3)}(X)}{-\log r}
    \leq C_2(r) \frac{N^{(2)}_{r}(X)}{-\log r}
    \leq C_3(r) \frac{N_{r/2}(X)}{-\log r/2}
    \leq C_4(r) \frac{N_{r/4}^{(1)}(X)}{-\log r/4}
  \]
  for some $C_1,C_2,C_3,C_4\to 1$ as $r\to 0$. Hence all the limits coincide.
\end{proof}
Note that the arguments above can be adjusted to replace open balls by closed ones and vice versa.
In fact, the box dimensions are invariant under taking closures.
In the sequel we will not refer to the alternative definitions by the superscript and use the
notation $N_r$ for any of these equivalent notions interchangeably.

\smallskip
There is one further alternative description of the box dimension when considering subsets of
$\R^d$ that uses the decay of Lebesgue measure $\cL^d$ in open $r$-neighbourhoods of sets. We will
denote the latter by $\langle F \rangle_{r} = \left\{ x\in \R^d \mid \inf_{z\in F}|z-x|<r
\right\}$ for $F\subseteq\R^d$.
\begin{proposition}
  Let $F \subset \R^d$ be bounded. Then
  \[
    \ubd F = \inf\left\{ s>0 \mid \limsup_{r\to0} r^{s-d}\cL^d\left( \langle F\rangle_{r}
    \right)=0\right\}
  \]
  and 
  \[
    \lbd F = \inf\left\{ s>0 \mid \liminf_{r\to0} r^{s-d}\cL^d\left( \langle F\rangle_{r}
    \right)=0\right\}.
  \]
\end{proposition}
A proof of this can be found in \cite[Proposition 2.4]{FalconerBible}.
The limits $\overline{\mathcal{M}}^s(F) = \limsup_{r\to0} r^{s-d}\cL^d\left( \langle F\rangle_{r}
\right)$ and $\underline{\mathcal{M}}^s(F) = \liminf_{r\to0} r^{s-d}\cL^d\left( \langle F\rangle_{r}
\right)$ are known as the upper and lower Minkowski content, respectively.

We note that the upper box dimension is finitely stable, $\ubd \bigcup_{i=1}^n X_i = \max_i
\ubd X_i$ but not countably stable. The lower box dimension fails to be finitely stable entirely.
Both dimensions are monotone $\ubd B \geq \ubd A$ and $\lbd B \geq \lbd A$ if $A\subseteq B$.

\subsubsection{The modified box dimension}
In many applications it is desirable to have countable stability of dimensions and the modified box
dimension remedies this issue. 
\begin{definition}
  Let $X$ be a totally bounded metric space. The \textbf{modified upper box dimension} is given by 
  \[
    \umbd X = \inf \left\{ \sup_{i\in\N} \ubd X_i \mid X = \bigcup_{i\in\N} X_i \right\}
  \]
  where the infimum is taken over all countable covers of $X$.
\end{definition}

%

\section{Picard's theorem: Relating box dimension and \texorpdfstring{$p$}{p}-variation}
\label{sect:picard}
In \cite{Picard08}, Picard showed that the variation index is related to
the box dimension of the associated real tree.
\begin{theorem}[Picard, {\cite[Theorem 3.1]{Picard08}}]\label{thm:picard}
  Let $f\in C(\bbS^1)$. If $f$ is non-constant, then $\ubd \cT(f)=I(f)$.
\end{theorem}
Our main contribution in this section is to provide a concise and direct proof of this statement
that uses standard arguments from fractal geometry and  dimension theory.
Before we do so, we remark that Picard theorem appears to be intricately connected with general
H\"older projections of the interval, see \cite{Balka23}.
\begin{proof}[Proof of \cref{thm:picard} (upper bound)]
  We first show that $\ubd \cT(f) \leq I(f)$.
  Let $\eps>0$ and choose $p$ and $s$ such that 
  \[
    \ubd \cT(f) - \eps < p <s < \ubd \cT(f).
  \]
  Then, by the definition of the upper box dimension, the maximal number $N_{r_j}=N_{r_j}(\cT(f))$ of
  disjoint centered closed balls
  of radius $r_j>0$ in $\cT(f)$ satisfies 
  $N_{r_j}(\cT(f)) \geq \max\{r_{j}^{-s},2\}$ for some sequence $r_j\to 0$ as $j\to\infty$.
  Let $\{B_k\}$ be such a maximal collection of disjoint closed $r_j$-balls in $\cT(f)$ with
  centres $\widetilde x_k \in\cT(f)$.
  For all $k\geq 1$, define
  $x_k, y_k\in [0,1)$ by
  \[
    x_k = \min\{ x \mid \pi(x) = \widetilde x_k\} \quad\text{ and }\quad
    y_k=\min\{y\mid{f}(z)\geq {f}(x_k)-r_j \text{ for all }z\in[y,x_k]\}.
  \]
  Note that this is well defined by the continuity of ${f}$ and that
  ${f}(x_k)-{f}(y_k)=r_j$ if $y_k \neq
  0$. Also observe that 
  $\pi(y_k) \in B_k$ and further that $\pi([y_k, x_k]) \in B_k$.
  Now, since the balls $B_k$ are disjoint, so are the intervals $[y_k, x_k]$ and there is at most
  one $k\in\N$ such that $y_k = 0$. We obtain
  \begin{align*}
    V^p(f) &= \sup\left\{ \sum_{i=1}^{n-1}|{f}(z_i) - {f}(z_{i+1})|^{p} \mid \{z_i\}_{i=1}^n \text{ is a finite
    partition of }[0,1] \right\} \\
    & \geq \sum_{k=1}^{N_{r_j}}|{f}(y_k) - {f}(x_k)|^p\\
    & \geq (N_{r_j}-1)r_j^p 
    \geq \tfrac12 r_j^{-s} r_j^p\to\infty \quad\text{as}\quad r_j\to0.
  \end{align*}
  Hence $I(f) \geq p > \ubd \cT(f) -\eps$ which completes the first part as $\eps>0$ was
  arbitrary.
\end{proof}
To prove the lower bound we need a technical lemma that proves the existence of a ``zig-zag''
subgraph. Observe first that restricting the variation to partitions with $M$ elements has the
effect of making the supremum a maximum:
Define
\[
  V_M^p(f) = \max\left\{ \sum_{i=1}^{M-1} |f(x_i)-f(x_{i+1})|^p \mid \left\{ x_i
  \right\}_{i=1}^M \text{ is a partition of }[0,1] \right\}.
\]
To see that the maximum exists, consider any 
sequence of tuples $\vec{x}(n) = \left\{ x_i(n) \right\}_{i=1}^M$ such that
\[
  \sum_{i=1}^{M-1} |f(x_i(n))-f(x_{i+1}(n))|^p \to V_M^p(f)
  \quad\text{as}\quad n\to\infty.
\]
Notice that $\vec{x}(n)\in [0,1]^M$ and by compactness, there exists a convergent subsequence with
limit $\vec{x}(\infty)=\{x_i(\infty)\}\in[0,1]^M$. Notice too, that by continuity of $\sum_{i=1}^{M-1}
|f(x_i)-f(x_{i+1})|^p$ as a function on $[0,1]^M$, we must have
\[
  \sum_{i=1}^{M-1} |f(x_i(\infty))-f(x_{i+1}(\infty))|^p = V_M^p(f)
\]
Hence, the supremum is achieved at the tuple $\vec x(\infty)$.

\smallskip
We say that a partition $\left\{ x_i \right\}$ of cardinality $M\in\N$ is 
\textbf{$(M,p)$-variation maximal} if 
\[
  \sum_{i=1}^{M-1} |f(x_i)-f(x_{i+1})|^p = V_{M}^{p}(f).
\]

\begin{lemma}\label{thm:zigzag}
  Let $f\in C(\bbS^1)$ be non-constant. Assume that $I(f)>1$.
  Then there exists an increasing sequence $(M_n)$ and sequence of partitions $\left\{
  x_i(n)\right\}_{i=1}^{M_n}$ that is $(M_n,p)$-variation maximal such that:
  \begin{enumerate}[(1)]
    \item $V_{M_n}^p(f)$ is 
      increasing,
    \item $V_{M_n}^p(f)\to V^p(f)$,

    \item $V_{M_n}^p(f) > V_{k}^p(f)$ for all $k<M_n$,

    \item and either
      \begin{equation}\label{eq:orderedPoints}
	f(x_1) < f(x_2) > f(x_3) < f(x_4) > \dots
	\quad\text{or}\quad
	f(x_1) > f(x_2) < f(x_3) > f(x_4) < \dots,
      \end{equation}
    \item as well as $d_f(x_i,x_{i+1}) = |f(x_i) - f(x_{i+1})|$ for all $1\leq i \leq M_n -1$.
  \end{enumerate}
\end{lemma}
\begin{proof}
  Since $f$ is non-constant, we have $V^p(f)>0$.
  Let $\eps>0$ be smaller than $V^p(f)$ and let $M=M(\eps)$ be the least integer such that 
  \[
    V_M^p(f) \geq \min\left\{ V^p(f)-\eps, 1/\eps \right\}.
  \]
  Let $\{x_i\}_{i=1}^M$ be an $(M,p)$-variation maximal partition.
  For $M=2$, this partition is given by taking a maximal and minimal argument of
  $f$. This can be seen to satisfy our conclusions for non-constant excursion functions. Hence we may
  assume that $M\geq 3$. 

  Further, we
  may assume that the string of inequalities in \eqref{eq:orderedPoints} holds, for if not,
  $f(x_{i-1})<f(x_i)<f(x_{i+1})$ or $f(x_{i-1})>f(x_i)>f(x_{i+1})$ for some $1<i<M$. Then 
  \begin{align*}
    |f(x_{i-1})-f(x_{i+1})|^p &= |f(x_{i-1})-f(x_i)+f(x_i)-f(x_{i+1})|^p\\
    &=\left( |f(x_{i-1})-f(x_i)| +|f(x_{i+1})-f(x_i)| \right)^p\\
    &\geq |f(x_{i-1})-f(x_i)|^p + |f(x_{i+1})-f(x_i)|^p,
  \end{align*}
  where the last inequality holds since $p\geq 1$.
  But then the partition $\left\{ x_j \right\}_{j=1}^M \setminus \left\{ x_i \right\}$ has variational sum at
  least as large as the partition $\left\{ x_j \right\}_{j=1}^M$ which contradicts minimal
  cardinality.

  \smallskip
  Having shown \eqref{eq:orderedPoints}, we now assume for a contradiction that there exists $i$
  such that $d_f(x_i,x_{i+1}) \neq |f(x_i) -
  f(x_{i+1})|$. Since
  \begin{align*}
    d_f(x_i,x_{i+1}) &= f(x_i)+f(x_{i+1}) - 2\min_{x_i \leq z \leq x_{i+1}}f(z)
    \\
    &\geq f(x_i)+f(x_{i+1})-2 \min\left\{ f(x_i),f(x_{i+1}) \right\} \\
    &= |f(x_i) -  f(x_{i+1})|,
  \end{align*}
  we must have $d_f(x_i,x_{i+1}) > |f(x_i) -  f(x_{i+1})|$.
  But then there exists $z$ with $x_{i} < z < x_{i+1}$ such that $f(z) < \min\left\{ f(x_i),
  f(x_{i+1})\right\}$. 

  We differentiate several cases. First assume $f(x_i) < f(x_{i+1})$ and $2\leq i\leq M-1$. Note that 
  \begin{equation}\label{eq:obs1}
    |f(x_{i-1})- f(x_i)|^p +|f(x_{i})- f(x_{i+1})|^p < |f(x_{i-1})-f(z)|^p + |f(z) -
    f(x_{i+1})|^p
  \end{equation}
  and so $\left\{ z \right\}\cup\left\{ x_j \right\}_{j=1}^M \setminus\left\{ x_i \right\}$ has variational sum larger
  than the partition $\left\{ x_j \right\}_{j=1}^M$ contradicting that $\{x_i\}_{i=1}^M$ was an
  $(M,p)$-variation maximal partition.

  Second, assume that $f(x_i)>f(x_{i+1})$ and $1\leq i \leq M-2$. Then,
  \begin{equation}\label{eq:obs2}
    |f(x_{i})- f(x_{i+1})|^p +|f(x_{i+1})- f(x_{i+2})|^p < |f(x_{i})-f(z)|^p + |f(z) -
    f(x_{i+2})|^p
  \end{equation}
  leading to the same contradiction with the partition $\left\{ z \right\}\cup\left\{ x_j
  \right\}_{j=1}^M \setminus\left\{ x_{i+1} \right\}$.

  The situations when $i=1$ and $i=M-1$ can be handled
  similarly, letting the summands with undefined function values in \eqref{eq:obs1} and \eqref{eq:obs2} be zero, respectively.

  \smallskip
  Note that the assumption that $I(f)>1$ implies that the function is non-smooth and not of bounded
  variation. Thus, increasing the number of partition elements must strictly increase the variation.
  Finally, letting $\eps\to 0$ gives a sequence of $M(\eps)$ with the desired properties.
\end{proof}
\begin{proof}[Proof of \cref{thm:picard} (lower bound)]
  We may assume that $I(f)>1$ as otherwise there is nothing to prove.
  Let $\eps>0$ and let $s,p$ be such that 
  \[
    \ubd \cT(f) < s < p < \ubd \cT(f) +\eps.
  \]
  \cref{thm:zigzag} above guarantees the existence of a subset $\left\{ x_i \right\}_{i=1}^{M_n}$ of
  minimal cardinality and maximal variational sum.
  We denote the variational sum by $V_{M_n}^p(f)$. By \eqref{eq:orderedPoints} the differences
  $f(x_i)-f(x_{i+1})$ alternate between positive and negative. By the pigeonhole principle either
  the positive or negative terms in the variational sum make up at least half the total sum
  $V_{M_n}^p(f)$.
  Assume for now that half the weight is given by the positive terms and that
  $f(x_1) < f(x_2)$.
  We obtain $\lfloor M_n/2\rfloor$ intervals $[x_{2i-1},x_{2i}]$ for which $f(x_{2i}) -
  f(x_{2i-1}) >0$.
  Write $\Delta_k = \left\{ x_{2i} \mid 2^{-(k+1)} < |f(x_{2i-1}) - f(x_{2i})| < 2^{-k} \right\}$.
  Observe that for all $i\neq j$ we have $d_f(x_{2i},x_{2j}) \geq \max\left\{
  d_f(x_{2i-1},x_{2i}), d_f(x_{2j-1},x_{2j})  \right\} \geq 2^{-(k+1)}$ since $f$ is increasing
  between $x_{2i-1}$ and $x_{2i}$, and $x_{2j-1}$ and $x_{2j}$, respectively. Hence
  $\pi(\Delta_k)\subset \cT(f)$ is a
  $2^{-(k+1)}$-separated set with respect to the $d_f$ metric. Note further that $\Delta_k=\varnothing$ for
  $k<k_0$ where $2^{k_0} \geq \|f\|_\infty$.
  So,
  \begin{align}
    V_{M_n}^p(f) &= \sum_{i=1}^{M_n-1}|f(x_i) - f(x_{i+1})|^p \leq 2 \sum_{i=1}^{\lfloor
    M_n/2\rfloor} |f(x_{2i}) - f(x_{2i-1})|^p\nonumber\\
    & = 2 \sum_{k=k_0}^\infty \sum_{x_{2i}\in\Delta_k}|f(x_{2i}) - f(x_{2i-1})|^p
    \leq 2\sum_{k=k_0}^\infty \#\Delta_k 2^{-p k},\label{eq:boundedvariation}
  \end{align}
  with a similar analysis holding for the remaining cases.
  Since $\pi(\Delta_k)$ is a $2^{-(k+1)}$-separated set in $\cT(f)$ and $\ubd \cT(f) <s$, we must
  have $\#\Delta_k \leq C 2^{s(k+1)}$ for some $C>0$.

  The bound in \eqref{eq:boundedvariation} then gives
  \[
    V_{M_n}^p(f) \leq 2\sum_{k=k_0}^{\infty}\#\Delta_k 2^{-p k} 
    =2 \sum_{k=k_0}^\infty C 2^{s(k+1) - p k} \leq C' 2^{(s-p)k_0} < \infty
  \]
  for some $C'>0$, making the bound independent of $M_n$.
  Taking limits gives $V^p(f) < \infty$ and so $I(f)\leq p$. Since $\eps>0$ was arbitrary,
  we get $I(f) \leq \ubd T(f)$ as required.
\end{proof}

\section{Discretised variations and variational contents}
\label{sect:morevariations}
The usage of $r$-separated subsets leads us to other natural definitions of variations that, to
the best of our knowledge, have not been considered in isolation yet.
We first introduce a discretised version, the $(p,r)$-variation:
\begin{definition}
  Let $f:[0,1]\to \R$. The \textbf{$(p,r)$-variation}
  $V_r^p(f)$ of $f$ is given by
  \begin{multline*}
    V_r^p(f)=\sup\Bigg\{ \sum_{i=1}^{n-1}|f(x_i)-f(x_{i+1})|^p \mid \left\{ x_i
    \right\}_{i=1}^n \text{ is a finite partition with }\\|f(x_i)-f(x_{i+1})| = r\Bigg\}
  \end{multline*}
\end{definition}
Observe that if $f$ is continuous the supremum must therefore be achieved for some
finite partition because of uniform continuity 
(fixing any first point gives finitely many other choices due to the exact $r$ step
gaps).
Further, since all summands are equal to $r$, the $(p, r)$-variation is equal to $r^p\cdot M_r^p(f)
$, where $M_r^p(f)$ is the maximal number of partition elements that satisfy the
definition.

This leads to the definitions of an upper and lower variation content.
\begin{definition}
  Let $f:[0,1]\to[0,\infty)$ be an excursion function. The \textbf{upper} and \textbf{lower variation
  content} is defined by 
  \[
    \overline{V}^p (f) = \limsup_{r\to0}V_r^p (f)
    \quad\text{and}\quad
    \underline{V}^p(f) = \liminf_{r\to0}V_r^p(f),
  \]
  respectively. 
  The \textbf{upper} and \textbf{lower variation index} are 
  \[
    \overline{I}(f) = \sup\Big\{p \geq 1 \mid \overline{V}^p(f) = \infty\Big\}
    \quad\text{and}\quad
    \underline{I}(f) = \sup\Big\{p \geq 1 \mid \underline{V}^p(f) = \infty\Big\}.
  \]
\end{definition}

This quantity does not have the property that the variation is always positive for non-constant
functions. In fact, the variation decays exponentially for exponents above its index.
\begin{lemma}\label{thm:threshold}
  Let $f:[0,1]\to\R$ and $s,p>0$.
  Then, the following two statements hold:
  \begin{enumerate}[(1)]
    \item If $\overline{V}^p(f)<\infty$, then there exists $C>0$ such that $V^s_r(f) \leq C
      r^{s-p}$.
    \item If $\overline{V}^p(f)>0$, then there exists $C>0$ and a sequence of scales $r_i\to 0$ such
      that $V^s_{r_i}(f) \geq C r_{i}^{s-p}$.
  \end{enumerate}
\end{lemma}
\begin{proof}
  Both statements follow from the simple observation that
  \begin{align*}
    V_r^s(f)
    &=
    \sup\left\{ \sum_{i=1}^{n-1}r^s \mid \left\{ x_i
    \right\}_{i=1}^n\text{ is a finite partition with }|f(x_i)-f(x_{i+1})| = r\right\}
    \\
    &=
    r^{s-p}\cdot\sup\left\{ \sum_{i=1}^{n-1}r^p \mid \left\{ x_i
    \right\}_{i=1}^n\text{ is a finite partition with }|f(x_i)-f(x_{i+1})| = r\right\}
    \\
    &=r^{s-p} \cdot V_r^p(f)
  \end{align*}
  combined with the facts that $\overline{V}^p(f)<\infty$ implies $\sup_{r>0}V^p_r(f)<\infty$ and
  $\overline{V}^p(f)>0$ implies the existence of scales $r_i\to 0$ such that $\inf_{i\in
  \N}V_{r_i}^p(f)>0$.
\end{proof}
We immediately get the following helpful proposition.
\begin{proposition}
  Let $f\in C(\bbS^1)$ be non-constant. Then for all $1\leq p<\overline{I}(f)$ we have
  $\overline{V}^p(f)=\infty$ and for all $\overline{I}(f) < p < \infty$, we have
  $\overline{V}^p(f)=0$. 

  Similarly, for all $1\leq p<\underline{I}(f)$ we have
  $\underline{V}^p(f)=\infty$ and for all $\underline{I}(f) < p < \infty$, we have
  $\underline{V}^p(f)=0$. 
\end{proposition}
Letting $p$ be the upper and lower variation index, the variation content can take any value, i.e.\
$\overline{V}^p(f),\underline{V}^p(f)\in[0,\infty)\cup\left\{ \infty \right\}$.

Our upper variation index coincides with the variation index.
\begin{lemma}
  \label{thm:upperandnormal}
  Let $f\in C(\bbS^1)$. Then 
  $\overline{V}^p(f) \leq V^p(f)$ and $\overline{I}(f) = I(f)$.
\end{lemma}
\begin{proof}
  Since 
  \begin{align*}
    V_r^p(f)&
    \leq\sup\left\{ \sum_{i=1}^{n-1}|f(x_i)-f(x_{i+1})|^p \mid \left\{ x_i
    \right\}_{i=1}^n\text{ is a finite partition}\right\} = V^p(f) 
  \end{align*}
  taking limits in $r$ immediately gives $\overline{V}^p(f)\leq V^p(f)$ and hence $\overline{I}(f) \leq
  I(f)$.

  It remains to show $\overline{I}(f) \geq  I(f)$. We may assume 
  $I(f)>1$ as otherwise there is nothing to
  show. Assume by way of contradiction that there exists an
  $p_0$ such that $\overline{I}(f) < p_0 < I(f)$. 
  Clearly,
  \begin{equation}\label{eq:firstRes}
    V^{p_0}(f)=\sup\left\{ \sum_{i=1}^{n-1}|f(x_i)-f(x_{i+1})|^{p_0} \mid \left\{ x_i
    \right\}_{i=1}^n \text{ is a finite partition } \right\} = \infty
  \end{equation}
  and by \cref{thm:threshold},
  \begin{equation}\label{eq:secondRes}
    V^{p_0}_r(f)=\sup\left\{ \sum_{i=1}^{n-1}r^{p_0} \mid \left\{ x_i
    \right\}_{i=1}^n\text{ is a finite partition with }|f(x_i)-f(x_{i+1})| = r\right\}\leq C r^\delta
  \end{equation}
  for some $C>0$ and $\delta = p_0 - \overline{I}(f)>0$.
  Recall that $M_r^{p_0}(f)$ is the number of elements in \eqref{eq:secondRes}. We can bound the
  expression by
  \[
    V_r^{p_0}(f) = M_{r}^{p_0}(f) r^{p_0} \leq C r^\delta.
  \]
  Note that $M_r^{p_0}(f)$ is non-decreasing as $r\to 0$.

  By \eqref{eq:firstRes} there exists a finite partition $\left\{ x_i \right\}_{i=1}^M$ such that 
  \begin{equation}\label{eq:thirdRes}
    \sum_{i=1}^{M-1}|f(x_{i}) - f(x_{i+1})|^{p_0} > \frac{C 2^{p_0+1}}{1-2^{-\delta}}\cdot\|f\|_{\infty}^{\delta}
  \end{equation}
  Partition the sum in \eqref{eq:thirdRes} into components of magnitude $2^{-k}\leq |f(x_i) -
  f(x_{i+1})| <  2^{-(k-1)}$. By maximality and monotonicity of $M_r^{p_0}(f)$,
  the number of components is bounded
  above by $M_{2^{-k}}^{p_0}(f)$. Hence, letting $k_0$ be the largest integer such that $2^{-k_0} \leq
  2\|f\|_{\infty}$ which gives a lower bound for the variation over a partition component, we obtain
  \begin{align*}
    \sum_{i=1}^{M-1}|f(x_{i}) - f(x_{i+1})|^{p_0}
    &=\sum_{k=k_0}^{\infty} \sum_{\substack{1 \leq i \leq M-1\\2^{-k}\leq |f(x_i) -
    f(x_{i+1})| <  2^{-(k-1)} }}|f(x_{i}) - f(x_{i+1})|^{p_0} \\
    &\leq \sum_{k=k_0}^{\infty}M_{2^{-k}}^{p_0}(f)\cdot 2^{-p_0 (k-1)}
    \leq C 2^{p_0}\sum_{k=k_0}^\infty  2^{-k\delta}\\
    &=C 2^{p_0} \frac{2^{-k_0 \delta}}{1-2^{-\delta}}
    \leq \frac{C 2^{p_0+1}}{1-2^{-\delta}}\cdot\|f\|_{\infty}^{\delta}.
  \end{align*}
  which directly contradicts \eqref{eq:thirdRes}.
\end{proof}

Applying Picard's theorem we obtain $\ubd \cT(f) = \overline{I}(f)$ and as we shall see below, the
lower index corresponds to the lower box dimension of the tree. This seems a very natural way of
correlating variation in terms of the box dimension of the associated space. It also gives a much
shorter proof of Picard's theorem by being more conceptual, using this notion of upper variation. 
\begin{proposition}\label{thm:comparable}
  Let $f\in C(\bbS^1)$ be non-constant.
  Then, for all $\eps>0$ there exists $r_0>0$ such that
  \begin{equation}\label{eq:comparable}
    \tfrac12 V_r^p(f) -\eps
    \leq
    N_r(\cT(f)) \cdot r^p 
    \leq 
    V_{r/4}^p(f) +\eps
  \end{equation}
  for all $0<r<r_0$.
  In particular
  \begin{equation}\label{eq:indexcompare}
    \lbd \cT(f) = \underline{I}(f)
    \quad
    \text{and}
    \quad
    \ubd\cT(f) = \overline{I}(f).
  \end{equation}
\end{proposition}
\begin{proof}
  Fix $0<r<1$ and let
  $\left\{ x_i \right\}$ be the maximal $V_r^p(f)$ partition of cardinality $M_r^p(f)$. Denote by
  $U,V\subset \left\{ 1,\dots,M_r^p(f)-1 \right\}$ the indices such that $f(x_i)=f(x_{i+1})-r$ for
  all $i\in U$ and $f(x_i) = f(x_{i+1})+r$ for all $i\in V$. We note that $U,V$ partition $\left\{
  1,\dots,M_{r}^{p}(f) \right\}$ and that the $U$ are all the indices in which an upcrossing of size
  $r$ appears.
  By the nature of upcrossings, we have $d_f(\pi x_{i+1}, \pi x_{j+1}) \geq r$ for all distinct
  $i,j\in U$. This is because 
  \begin{align*}
    d_f\left( \pi x_{i+1}, \pi x_{j+1} \right)
    &\geq
    \begin{cases}
      d_f\left( \pi x_{j}, \pi x_{j+1}\right) &\text{if }i<j,\\
      d_f\left( \pi x_{i}, \pi x_{i+1}\right) &\text{ otherwise.}
    \end{cases}
    \\
    &=r.
  \end{align*}
  Thus, $\bigcup\left\{ \pi x_{i+1} \mid i\in U \right\}$ is an $r$-separated subset of $\cT(f)$ and
  $N_r(\cT(f)) \geq \# U$.

  Note further that $V_r^p(f) = (\#U+\#V)r^p$ and $ (\#U-\#V)r \geq\|f\|_{\infty}-r$.
  Combining these two bounds, we obtain
  \begin{align*}
    V^p_r(f) 
    &\leq r^p\left( 2\#U  - r^{-1}\|f\|_{\infty} +1 \right)
    \intertext{and so}
    N_r(\cT(f)) \geq \#U& \geq
    \tfrac12 \left(r^{-p}V_r^p(f)+r^{-1}\|f\|_{\infty}-1\right)
    \geq \tfrac12 r^{-p} V_r^p(f) -\tfrac12.
  \end{align*}

  \smallskip
  For the other inequality, let $\{\tau_1,\dots, \tau_N\}\in\cT(f)$ be an $r$-separated set, where
  we write $N=N_r(\cT(f))$.
  Set $x_i = \min\left\{ \pi^{-1} \tau_i \right\}$. Without loss of generality, reindexing if
  necessary, we may assume $0\leq x_1 < x_2 < \dots <  x_{N} < 1$.
  Consider $I_i = \left\{ f(x) \mid x\in [x_i,x_{i+1}] \right\}$, where $1\leq i < N$.
  By continuity, $I_i$ is a closed interval. Since $d(\tau_i,\tau_{i+1})\geq r$, we must further
  have that $f(x_i)+f(x_{i+1})-2\min_{z\in[x_i, x_{i+1}]}f(z) \geq r$. Hence, $I_i$ is an interval
  of length $|I_i| \geq r/2$ and there exist a $k_i\in\N_0$ such that $\tfrac{k_i r}{4}\in I_i$ and
  $\tfrac{(k_i+1)r}{4}\in I_i$. 
  Finally, we write $I_0 =\left\{ f(x) \mid x\in [0,x_{1}] \right\}$ and $I_N =\left\{ f(x) \mid
  x\in [x_{N},1] \right\}$ (which may be singletons).

  We now inductively choose a partition of $[0,1]$. Let $z_{0,0} = 0$ and set $z_{0,1},
  z_{0,2},\dots,z_{0,m_0}\in  I_0$ such that $z_{0,n}<z_{0,i+1}$ and $f(z_{0,i}) = i\cdot \tfrac{r}{4}$.
  If $I_0$ has length $|I|<\tfrac14$, we get $m_0=0$. 
  Note that this is well defined by continuity of $f$.
  We write $z_0 = \max\left\{ a_0,\dots,a_m \right\}$.
  Having defined $z_{n-1,m_{n-1}}$, we define $m_n\geq 1$ and $z_{n,1},z_{n,2},\dots,z_{n,m_n}$, in the
  following way:
  Let 
  \[
    z_{n,1}= \min\left\{ x\in[x_n,x_{n+1}] \mid |f(x)-f(z_{n-1,m_{n-1}})| = r/4 \right\}
  \]
  and inductively define 
  \[
    z_{n,l+1} = \min\left\{ x\in[z_{n,l}, x_{n+1}] \mid  |f(x)-f(z_{n,l})| =
    r/4 \right\}.
  \]
  By uniform continuity the inductive process (in $l$) will eventually halt. We write $m_n$ for the maximal such index
  and note that $m_n\geq 1$ since $k_n$ and $k_{n}+1$ are in $I_n$.
  Similarly, the induction in $n$ will end with $n=N-1$.

  We relabel the sequence $(z_{0,0},\dots,z_{0,m_0},z_{1,1},\dots,z_{1,m_1},\dots,z_{N-1,m_{N-1}})$
  by $z_1, z_2,\dots, z_K$ and note that by construction $z_i<z_{i+1}$ and $|f(z_i)-f(z_{i+1})| =
  r/4$.
  Hence,
  \begin{align*}
    V_{r/4}^p (f)& \geq \sum_{i=1}^{K-1} |f(z_i)-f(z_{i+1})|^p
    \\
    &\geq\sum_{i=1}^{N-1}\sum_{\substack{z_j,z_{j+1}\in I_i\\f(z_j)=k_i \text{ or }f(z_{j+1})=k_i}} |f(z_i)-f(z_{i+1})|^p
    \\
    &= (N_r(\cT(f))-1)\left(\frac{r}{4}\right)^p
  \end{align*}
  from which \eqref{eq:comparable} follows. 

  The conclusion \eqref{eq:indexcompare} then follows upon taking limits.
\end{proof}

We note that the discrepancy between the variations of scales $r$ and $r/4$ can be overcome directly
in doubling spaces.
\begin{corollary}\label{thm:doublingComparable}
  Let $f\in C(\bbS^1)$ be non-constant and $\cT(f)$ be a doubling space.
  There exists a universal constant $C>0$ such that for all $\eps>0$ there exists $r_0>0$ with
  \[
    C^{-1}V_r^p(f) -\eps \leq N_r(\cT(f))\cdot r^p \leq C V_{r}^p(f)+\eps
  \]
  for all $0<r<r_0$.
\end{corollary}




\paragraph{Remark.}
Finally, we remark that the $(p,r)$ variation is not subadditive.
However, it is quasi-subadditive up to a
multiplicative constant, when rescaling. That is, 
\[
  V_r^p(f+g) \leq C \left( V_{r/2}^{p}(f) + V_{r/2}^p(g) \right)
\]
for some fixed $C>0$ 
Taking limits in $r$, we see that 
\begin{equation}\label{eq:quasinorm}
  \overline{V}^p(f+g) \leq
  C\left(\overline{V}^p(f)+\overline{V}^{p}(g)\right).
\end{equation}
It follows that $\overline{V}^{p}(f)$ is a quasi-norm on the space of functions;
\cref{eq:quasinorm} is also known as the relaxed triangle inequality, see \cite{Greenhoe16}.
\section{Dimensions of Graphs and Trees}\label{sect:shmerkin}
The study of the dimension theory of graphs $\Gamma(f) = \left\{ (x,f(x)) : x\in[0,1] \right\}$
has a long history dating back to Besicovitch and Ursell \cite{Besicovitch37} with considerable
interest in recent years \cite{BerryLewis80, MauldinWilliams86b,Moreira94, Hunt98, Bedford89, Urbanski90,
Allaart20,Baranski14, Barany18}.
See also the book by Massopust \cite{MassopustBook}

We say that $f$ is \textbf{$\alpha$-anti-H\"older} if there exists $C>0$  such that for all
$0<\delta<1$ and $x\in [0,1]$ there exists $y\in B(x,\delta)$ such that 
\[
  |f(x) - f(y)| \geq C \delta^{\alpha}.
\]
The following proposition is well-known. A proof can be found in \cite[Corollary 11.2]{FalconerBible}.
\begin{proposition}
  Let $f:[0,1]\to\R$ be continuous.
  \begin{enumerate}[(1)]
    \item Assume $f$ is $\alpha$-H\"older. Then, $\ubd \Gamma(f) \leq 2-\alpha$.
    \item Assume $f$ is $\alpha$-anti-H\"older. Then $\lbd \Gamma(f) \geq 2-\alpha$.
  \end{enumerate}
\end{proposition}
In fact, more is true. 
\begin{proposition}\label{thm:uppergraphbound}
  Let $f:[0,1]\to\R$ be continuous.
  Then, 
  \[
    \ubd\Gamma(f) \leq 2-\frac{1}{I(f)}.
  \]
\end{proposition}
We refer the reader to \cite{Norvaisa02} for a proof.
See also  \cite{Dubuc96,Dubuc89,Manstavicius05,Moreira94,Normant91,Tricot89,Tricot88}
for some further history linking variation methods to the upper box dimension.

Recalling Picard's theorem (\cref{thm:picard}), we may replace the variation index by the dimension of
the tree.
We obtain
\[
  \ubd \Gamma(f) \leq 2-\frac{1}{\ubd \cT(f)} 
  \quad
  \Longleftrightarrow
  \quad
  \ubd \cT(f) \ge \frac{1}{2-\ubd \Gamma(f)}.
\]
Recall that the tree $\cT(f)$ is invariant under time changes and so, 
\[
  \ubd\cT(f) = \ubd \cT(f\circ \tau) \geq \frac{1}{2-\ubd\Gamma(f\circ\tau)}
\]
for all $\tau \in \Homeo(\bbS^1)$.

Our main result in this section is that there exists a variational principle that obtains this
bound.
\begin{theorem}\label{thm:shmerkin}
  Let $f\in C(\bbS^1)$ be non-constant.
  Then there exists $\overline\tau \in \Homeo(\bbS^1)$ such that
  \[
    \ubd \Gamma(f\circ \overline\tau) = 2-\frac{1}{\ubd \cT(f)} = 2-\frac{1}{I(f)}. 
  \]
  In particular, for all $\tau\in\Homeo(\bbS^1)$,
  \[
    \frac{1}{2-\ubd\Gamma(f\circ \tau)} 
    \quad\leq\quad
    \ubd \cT(f)
    \quad=\quad
    \max_{\kappa\in\Homeo(\bbS^1)}\; \frac{1}{2-\ubd\Gamma(f\circ \kappa)}.
  \]
\end{theorem}

While we are unable to ascertain what notion (if any) in the tree corresponds to
\[
  \min_{\tau\in\Homeo(\bbS^1)}(2-\ubd\Gamma(f\circ \tau))^{-1},
\]
we will show that there exists a
time change for which the modified upper box dimension gives a lower bound.
\begin{corollary}
  \label{thm:modified}
  Let $f\in C(\bbS^1)$ be non-constant. Then there exists a time change
  $\underline\tau\in\Homeo(\bbS^1)$ such that 
  \[
    \ubd \Gamma (f\circ\underline\tau) \geq 2-\frac{1}{\umbd \cT(f)}.
  \]
\end{corollary}

\paragraph{Organisation.}
The rest of this Section is organised as follows. We first construct the maximising time change
$\overline \tau$ in \cref{sect:timechange} and then prove that $\ubd\Gamma(f\circ \overline\tau)\geq
2- (\ubd \cT(f))^{-1}$ in \cref{sect:shmerkinproof}.
Finally, in \cref{sect:modproof} we construct the time change under which the modified upper box
dimension is a lower bound, proving \cref{thm:modified}.

\subsection{A maximising time change}
\label{sect:timechange}
In this section we inductively construct a time change $\overline\tau \in \Homeo(\bbS^1)$ for which 
\[
  \ubd\cT(f) = I(f) = \frac{1}{2-\ubd\Gamma(f\circ\overline\tau)}.
\]

Let $p = \ubd\cT(f) $. Let $(\eps_n)_{n=1}^{\infty}$ be a decreasing sequence of positive reals such that $\eps_n \to 0$.
Let 
\[
D(r) = \left\{ [x_i,y_i] \subseteq[0,1] \mid |f(y_i) - f(x_i)| = r; x_i<y_i<x_{i+1}; 0\notin[x_i,y_i] \right\}\]
be the set of intervals with disjoint interior that do not contain $0$ and in which there is an upcrossing or downcrossing.
By convention we will order these subsets by stipulating that $x_i<y_i\leq x_{i+1}$ when referring to elements in $D(r)$.
Since $\overline{I}(f) = p$ by \cref{thm:upperandnormal}, there exists a sequence of $(r_n)_{n=1}^{\infty}$ such that 
$\#D(r_n) = \lfloor r_n ^{-p(1-\eps_n)}\rfloor$.
By taking subsequences $k_n$ if necessary, we may assume without loss of generality that
\begin{enumerate}[(i)]
  \item\label{it:primeBound} $r_{k_{n+1}}^{-p(1-\eps_{n+1})} \geq \# D(r_{k_{n+1}}) - \# D(r_{k_n}) \geq r_{k_{n+1}}^{-p(1-2\eps_{n+1})}>2$,
  \item\label{it:iterBound} $r_{k_{n+1}}^{2\eps_{n+1}} \leq r_{k_n}$,
  \item\label{it:minDecrease} $r_{k_n} \leq \exp(-k_n^2)$.
\end{enumerate}
We shall relabel $r_{k_n}$ by $r_n$ to avoid unnecessarily cumbersome notation.

\medskip 
We construct $\overline\tau$ by first constructing its inverse $\phi$ as the pointwise limit of homemorphisms $\phi_n$, which are defined on an increasing sequence of discrete sets $F_n\subset [0,1]$ that contain the endpoints of intervals in $D(r_j)$ for $j\leq n$. Further, we will show that 
\begin{equation}\label{eq:inductiveHyp}
  \diam(\phi_n(J))\geq \tfrac1{4^n} r_n^p
  \qquad
  \text{for all $J\in D(r_n)$.}
\end{equation}

Assume for the time being that $f$ is nowhere constant. 
Let $\phi_0 = \id$, $F_0=\{0\}$ and
\[
  F_1 = F_0 \cup \left\{ x_i \mid [x_i,y_i]\in D(r_1) \right\}\cup\left\{ y_i \mid [x_i,y_i]\in
  D(r_1) \right\}.
\]
Recall that we index elements of finite subsets of $\bbS^1$ with respect to the ordering $<$ on
$[0,1)$.
We set
\(
  \phi_1(z_i) = \frac{i}{\#D(r_1)+1}
\)
for all $z_i\in F_1\setminus F_0$ and extend $\phi_1$ to $F_1$ by letting $\phi_1|_{F_0} = \phi_0|_{F_0}$.
Having defined $\phi_1|_{F_1}$, we define $\phi_1$ to linearly interpolate for all intermediate values,
that is for $z\in [x_i,y_i] \in D(r_1)$, we set $\phi_1(z) = \frac{\phi_1(y_i)-\phi_1(x_i)}{y_i-x_i} (z - x_i) + \phi_1(x_i)$, for $z\in[y_i,x_{i+1}]$ we set $\phi_1(z) = \frac{\phi_1(x_{i+1})-\phi_1(y_i)}{x_{i+1}-y_i} (z - y_i) + \phi_1(y_{i})$ and similarly for $z\in[0,x_1]$ and $z\in[y_n,1]$ where $y_n = \max F_1\setminus F_0$.
Note that $\phi_1([x_i,y_i])$ are intervals of diameter
\[
  \diam(\phi_1([x_i,y_i])) \geq \frac{1}{2\#D(r_1)+1} \geq \tfrac14 r_1^{p(1-\eps_1)}\geq\tfrac14 r_1^p
\]
which establishes \eqref{eq:inductiveHyp} for $n=1$.
\medskip

Having defined $\phi_{n-1}$ and $F_{n-1}$, we set 
\[
  F_n = F_{n-1} \cup \left\{ x_i \mid [x_i,y_i] \in D'(r_n) \right\}  \cup \left\{ y_i \mid [x_i,y_i]\in D'(r_n) \right\}
\]
where $D'(r_n) = \left\{ [x_i, y_i]\in D(r_n) \mid [x_i,y_i] \cap F_{n-1} = \varnothing \right\}$.
Now $\#D'(r_n) \leq \#D(r_n) \leq r_{n}^{-p(1-\eps_n)}$
and by \eqref{it:primeBound},
$\#D'(r_n)\geq \#D(r_n)-\#D(r_{n-1}) \geq r_{n}^{-p(1-2\eps_n)}$.
Now set $\phi_n|_{F_{n-1}} = \phi_{n-1}|_{F_{n-1}}$ and equidistribute images of $F_{n}\setminus F_{n-1}$
within each interval $(a_i, a_{i+1})$, where $a_i,a_{i+1}\in \phi(F_{n-1})$ and $0=a_0 < a_1 < \dots < 1$.
Note that \eqref{it:iterBound} implies $r_{n-1}^{p}r_{n}^{-2p\eps_n} \geq 1$. 
Thus, \eqref{it:primeBound}, \eqref{it:iterBound}, and the inductive hypothesis
\eqref{eq:inductiveHyp} for $n-1$ gives, for elements $[x_i,y_i]\in D'(r_n)$,
\begin{align} 
  \diam(\phi_n([x_i,y_{i}])) 
  &\geq \frac{\min_{i\neq j} |a_{i}-a_j|}{2\#D'(r_n)+1}
  \geq \tfrac14 r_n^{p(1-2\eps_n)} \cdot \min_{i\neq j} |a_{i}-a_j|
  \nonumber\\
  &\geq \tfrac14 r_n^{p(1-2\eps_n)}\tfrac{1}{4^{n-1}} r_{n-1}^p
  = \tfrac{1}{4^n} r_n^p  r_{n-1}^p r_n^{-2p\eps_n}
  \geq \tfrac{1}{4^n} r_n^p.
  \label{eq:lengthBound}
\end{align}
This shows \eqref{eq:inductiveHyp} for $n$ and completes the induction step.

\medskip
To construct $\phi$, we first note that $\phi_n|_{F_k} = \phi_k|_{F_k}$ for all $k\geq n$.
Write $F = \bigcup_{k\in\N}F_k$.
For all $x\in \bigcup_{k\in\N} F_k$ we set $\phi(x) = \phi_m(x)$ where $m$ is the minimal integer such that $x\in F_m$.
It immediately follows that $\phi$ is non-decreasing on $F$.
Since $f$ is nowhere constant, for any $[a,b]\subset \bbS^1$ there exists large enough $n$ such that there exists $[x_i,y_i]\in D'(r_n)$ contained in $[a,b]$. Hence, $F$ is a countable dense subset of $\bbS^1$.
Thus, setting $\phi(x) = \lim_{\substack{y\searrow x , y\in F}}\phi(y)$ for $x\in \bbS^1\setminus F$ and using monotonicity, $\phi$ is a well-defined cadlag function on $\bbS^1$.
By denseness of $F$ in $\bbS^1$, we further see that $\phi$ is (strictly) increasing on $[0,1)$, since for any $a,b\in[0,1)$ there exists $n$ large and an interval $[x_i,y_i]\in D'(r_n)$ contained in $[a,b]$ and
\[
  \phi(a) \leq \phi(x_i) = \phi_n(x_i) < \phi_n(y_i) = \phi(y_i) \leq \phi(b).
\]
Finally, we show that $\phi$ is continuous.
To do so, we first establish that $\phi(F)$ is dense in $\bbS^1$. 
Assume for a contradiction that $\phi(F)$ is not dense.
Then there exists an interval $[a,b]\subset[0,1)$ of length $\delta=b-a$ such that
for all $\eps>0$, we have $[a,b]\cap \phi(F) = \varnothing$, $[a-\eps,a]\cap \phi(F) \neq \varnothing$
and $[b,b+\eps]\cap \phi(F) \neq \varnothing$.
Let $\eps<\tfrac14 \delta$ and choose $n$ large enough such that there exist consecutive points $z_i,z_{i+1}\in\phi(F_n)$
with $z_i \in [a-\eps,a]$ and $z_{i+1}\in[b,b+\eps]$. 
Since $F$ is dense there must exist a minimal $k>n$ and an interval
$[x_j,y_j]\in D'(r_k)$ with $\phi([x_j,y_j]) \subset (z_i,z_{i+1})$.
Thus, since $\phi_k$ evenly spreads points, the set $F_k \cap [z_i, z_{i+1}]$ has at least four elements separated by at most $(z_{i+1}-z_i)/3 \leq (\delta+2\tfrac{\delta}{4}) /3 = \tfrac12 \delta $.
But then $[a,b]\cap F_k$ is non-empty, a contradiction.

Since $\phi:\bbS^1\to\bbS^1$ is a strictly increasing bijection from a dense subset of $\bbS^1$ onto a dense subset of $\bbS^1$, it is continuous and hence a homeomorphism. 
Defining $\overline\tau = \phi^{-1}$ gives a mapping $\overline\tau\in\Homeo(\bbS^1)$ that we will use as a time-change.

We may avoid the assumption that $f$ is nowhere constant by realising that for non-constant $f$, the
construction above is well-defined on the closure $\cl(F)$ and $[0,1)\setminus\cl(F)$ is either
empty (and $f$ is nowhere constant) or is a countable union of disjoint intervals $J_i$.
We may then define
\begin{equation}\label{eq:integralstretch}
  \phi_0(x) = \tfrac12 \phi(x) + \frac{1}
  {2\cL^1([0,1)\setminus\cl(F))} \int_0^x \mathbf{1}_{[0,1)\setminus\cl(F)}(y)dy
\end{equation}
which can likewise be shown to be a homeomorphism on $\bbS^1$ that satisfies \eqref{eq:lengthBound}
up to another multiplicative constant of $\tfrac12$.

\subsection{Proof of Theorem \ref{thm:shmerkin}}
\label{sect:shmerkinproof}
Recall that the upper $s$-Minkowski content of $X\subset \bbS^1\times\R$
is given by 
\[
  \overline{\mathcal{M}}^s (X) = \limsup_{r\to 0}r^{s-2} \cL^2(\langle X \rangle_r),
\]
where $\cL^2$ is the $2$-dimensional Lebesgue measure.
The upper box dimension may be estimated from below by finding $t>0$ such that
$\overline{\mathcal{M}}^t (X) > 0$.
If $X=\Gamma(g)$ is the graph of a continuous function $g:\R\to\bbS^1$ we can bound $\langle X
\rangle_r$ from below by 
\[
  U(r) = U_g(r)= \bigcup_{x\in\bbS^1} (x-r,x+r)\times \{g(x)\}.
\]
That is,
if $\limsup_{r\to 0}r^{t-2} \cL^2(U_g(r)) > 0$, then $\ubd \Gamma (g) \geq t$.

\medskip
Consider $A(r_n) = \phi(D'(r_n))$ for $\phi$ and $D'$ as in the last section. 
The collection $A(r_n)$ consists of at least $r_n^{-p(1-2\eps_{n})}$ intervals with disjoint
interior, each of length at least $\tfrac18 r_n^p$, by \eqref{eq:lengthBound} and the comment after
\eqref{eq:integralstretch}.
Notice further that for all $[a_i,b_i]\in A(r_n)$, we get $f\circ\overline\tau([a_i,b_i]) = f([x_i,y_i])$ for some $[x_i,y_i]\in D'(r_n)$ and thus $|f\circ\overline\tau(a_i) - f\circ\overline\tau(b_i)| = r_n$.
Hence, 
\[
  \cL^2(U_{f\circ\overline\tau}(\tfrac{1}{8^n} r_n^p)) \geq \sum_{[a_i,b_i]\in A(r_n)} \tfrac{1}{8^n}
  r_n^p \cdot r_n
  \geq
  r_n^{-p(1-2\eps_{n})}\cdot\tfrac{1}{8^n} r_n^p \cdot r_n
  =\tfrac{1}{8^n} r_n^{1+2p\eps_n}.
\]
Let $0<\delta<p-1$ be arbitrary. 
We obtain, for $s = 2-\tfrac{1+\delta}{p}>1$,
\begin{align*}
  \limsup_{\rho\to 0}\rho^{s-2} \cL^2(U_g(\rho))
  &\geq
  \limsup_{n\to\infty} (\tfrac{1}{8^n} r_n^{p})^{s-2}\tfrac{1}{8^n} r_n^{1+2p\eps_n}
  \\&=
  \limsup_{n\to\infty}(\tfrac{1}{8^n})^{s-1}\cdot r_{n}^{ps-2p+1+2p\eps_n}
  \\&=
  \limsup_{n\to\infty} e^{-\alpha n} \cdot r_{n}^{2p\eps_n-\delta}
  &&\text{for some $\alpha>0$}
  \\&\geq
  \limsup_{n\to\infty} e^{-\alpha n} \cdot r_{n}^{-\delta/2}
  \\&\geq
  \limsup_{n\to\infty} e^{-\alpha n} \cdot e^{(\delta/2) n^2}
  &&\text{by \eqref{it:minDecrease}}
  \\
  & =\quad \infty
\end{align*}
and hence $\ubd \Gamma(f\circ\overline\tau) \geq 2-\tfrac1p-\tfrac\delta p$ for all $\delta>0$.
The conclusion to \cref{thm:shmerkin} follows immediately upon application of the general upper
bound in \cref{thm:uppergraphbound}. \qed


\subsection{Proof of Corollary \ref{thm:modified}}
\label{sect:modproof}
The proof follows as a consequence from the exact time changes in \cref{thm:shmerkin}.
Let $T_i$ be a countable decomposition of $\cT(f)$ into disjoint sets.
Each preimage $\pi^{-1}T_i = J_i$ corresponds to a countable union of disjoint intervals which may
be open, half-open, or closed. Without loss of generality we may assume that all $T_i$ are
connected. Since every restriction to $J_i$ corresponds to a tree, we have
$\ubd \Gamma(f|_{J_i}) \leq 2-(\ubd \cT(f|_{J_i}))^{-1}$.
By \cref{thm:shmerkin} there exists a time change $\tau_i$ such that $\tau_i|_{\bbS^1\setminus
J_i}=\id$ and 
\[
  \ubd \Gamma(f\circ \tau_i|_{J_i}) = 2-\frac{1}{\ubd \cT(f|_{J_i})} = 2-\frac{1}{\ubd T_i}.
\]
Since $\tau_i$ fixes endpoints, the composition $\tau = \tau_1\circ\tau_2\circ\dots$
is a homeomorphism on $\bbS^1$.
By monotonicity of the box dimension we have
\[
  \ubd \Gamma(f\circ\tau) \geq
  \ubd \Gamma(f\circ\tau|_{J_i}) = 2-\frac{1}{\ubd T_i}
\]
for all $i$ and so
\[
  \ubd \Gamma(f\circ\tau)\geq \sup_i \left( 2-\frac{1}{\ubd T_i} \right)
  =2-\frac{1}{\sup_i \ubd T_i}
  \geq 2-\frac{1}{\umbd \cT(f)}
\]
as $\umbd\cT(f) \leq \sup_i \ubd T_i$. \qed

\section{Examples and Counterexamples}
\label{sect:examples}
\subsection{Value of upper variation content}
We will first create a general set up with which we give several examples of functions.
Let $f(x)$ be the ``zig-zag'' function consisting of $K_n$ peaks of height $h_n$ in the interval
$[2^{n},2^{n-1}]$. That is, for all integers $0\leq k \leq 2K_n$,
\[
  f(2^{-n}+\tfrac{j}{2K_n}2^{-n}) = \begin{cases}
    0 & k \text{ is even,}\\
    h_n & k\text{ is odd.}
  \end{cases}
\]
and interpolating linearly in-between.
\begin{example}
  Let $f$ be as above. Let $p>1$, $h_n=2^{-n}$, and $K_n = 2^{pn}$. Then,
  $\overline{V}^p(f)=\tfrac{2^{1+p}}{2^p-2}$ and $V^p(f)=\infty$.
\end{example}
We first calculate $V_r^p(f)$ for $r=2^{-k}$.
Its value is given by 
\[
  V_{2^{-k}}^p(f) = \sum_{n=1}^k 2 K_n \frac{2^{-n}}{2^{-k}}\cdot 2^{-pk}
  = 2^{1+(1-p)k} \sum_{n=1}^k 2^{pn}2^{-n}
  =\frac{2^{1+p}}{2^p-1}.
\]
It is readily seen that $V_r^p(f)$ cannot exceed this value for any $r>2^{-k}$ and hence
$\overline{V}^p(f) = \tfrac{2^{1+p}}{2^p-2}$.
For the $p$-variation, we can count terms individually up to some size $2^{-k}$. We get
\[
  V^p(f) \geq \sum_{n=1}^k 2 K_n 2^{-pn} = 2\sum_{n=1}^{k} 1 = 2k
\]
for any $k\in\N$. Hence $V^p(f) = \infty$.

\bibliographystyle{alpha}
\setlength{\parskip}{-2pt}
\footnotesize
\bibliography{lib}{}

\end{document}